\documentclass{amsart}
\usepackage{graphicx} 

\input{header}

\title{On stable Kim-forking and rosy theories}
\author{Alberto Miguel-Gómez}
\date{31 March 2025}
\address{Department of Mathematics, Imperial College London, London SW7 2AZ, UK}
\email{a.miguel-gomez22@imperial.ac.uk}

\begin{document}

\maketitle
\begin{abstract}
    We provide a partial answer to a question asked independently by Kim and d'Elbée and show that, under the assumption of the stable Kim-forking conjecture, every $\NSOP_1$ rosy theory must be simple. We also prove that the theory of a Frobenius field has stable Kim-forking.
\end{abstract}

\section{Introduction}
Two intensely studied generalisations of the class of simple theories are that of rosy theories (introduced by Onshuus in \cite{onshuus2006properties} and developed further by Onshuus and Ealy in \cite{onshuus2003th} and \cite{ealy2007characterizing}) and that of $\NSOP_1$ theories (introduced by D\v{z}amonja and Shelah in \cite{dzamonja2004maximality}). A common characteristic of both theories is the presence of a ``canonical'' independence relation whose properties determine the class. The connection of these classes with notions of independence was something present in the study of rosy theories from the start, but appeared in the context of $\NSOP_1$ theories more recently, first in the work of Chernikov and Ramsey in \cite{chernikov2016model}, and then explicitly in terms of the notion of Kim-independence in several papers of Kaplan, Shelah, and Ramsey in \cite{kaplan2020kim}, \cite{kaplan2019local}, and \cite{kaplan2021transitivity}. 

The following, natural question regarding the relationship between these two classes was posed by Byungham Kim and Christian d'Elbée, independently of each other, in 2021:
\begin{question}[\protect{cf., \cite{kim2021weak}, \cite{d2021forking}}]
    Is there an $\NSOP_1$ rosy theory which is not simple?
\end{question}
In this paper, we show that an answer to this question can be obtained in connection with another conjecture coming out of the study of simple theories.

Recall that we say a formula $\phi(x,y)$ is \textbf{stable} if there are no sequences $(a_i)_{i<\omega}$ and $(b_i)_{i < \omega}$ such that, for all $i,j < \omega$, we have $\models \phi(a_i, b_j)$ iff $i < j$. In analogy to the stable forking conjecture proposed by Hart, Kim, and Pillay (cf., \cite{kim2001simplicity}), and which also relates to the dependent dividing conjecture stated by Chernikov (cf., \cite[Proposition 4.14]{chernikov2014theories}), Bossut has recently introduced, in \cite{bossut2024stable}, what we call here the \textit{stable Kim-forking conjecture}. Our main result (Theorem \ref{thm:stable-kim-forking-nsop1-and-rosy-is-simple}) says that, if the stable Kim-forking conjecture holds, every $\NSOP_1$ rosy theory $T$ must be simple. In fact, we prove something slightly stronger, since we only require that $T^\eq$ has stable Kim-forking over models.   

The structure of the paper is as follows. After some notation and terminology in \S2, we introduce in \S3 the stable Kim-forking conjecture and discuss several important $\NSOP_1$ examples where it has been verified. We include in that section a proof that the conjecture holds in the theory of a Frobenius field (Proposition \ref{prop:omega-free-pac-has-stable-kim-forking}). Then, in \S4, we prove our main result and, as an application, we conclude that, if the stable Kim-forking conjecture holds, there is no strictly $\NSOP_1$ pregeometric theory with geometric elimination of imaginaries. 
\subsection*{Acknowledgments}
The present work was completed during my PhD at Imperial College London supported by an EPSRC scholarship. I am deeply grateful to my supervisors, David Evans and Charlotte Kestner, for their guidance, suggestions, and constant support. The contents of the paper were developed during a research visit at the University of Notre Dame. I thank Anand Pillay and Nicholas Ramsey for many useful comments and improvements, as well as Atticus Stonestrom for drawing my attention to the class of pregeometric theories.
\section{Preliminaries}

\subsection{Abstract independence relations}
Let $T$ be a complete theory with infinite models and $\M \models T$ a monster model. For us, an \textbf{independence relation} $\ind$ is an $\Aut(\M)$-invariant ternary relation on subsets of our fixed monster model, i.e., $A \ind_C B$ (read as ``$A$ is independent from $B$ over $C$'') iff $\sigma(A) \ind_{\sigma(C)} \sigma(B)$ for all $\sigma \in \Aut(\M)$. Many abstract properties that independence relations may satisfy appear already in \cite{shelah1978classification}, and more explicitly, in \cite{baldwin1988fundamentals}. We follow Adler's modern presentation in \cite{adler2009geometric}:
\begin{itemize}
    \item \textit{Right monotonicity}: For all $A, B, C \subseteq \M$, if $A \ind_C B$ and $B' \subseteq B$, then $A \ind_C B'$.
    \item \textit{Right base monotonicity}: For all $A \subset \M$ and $D \subseteq C \subseteq B \subset \M$, if $A \ind_D B$, then $A \ind_C B$.
    \item \textit{Right extension}: For all $A, B, C \subseteq \M$, if $A \ind_C B$ and $B \subseteq B' \subset \M$, there is some $A' \equiv_{BC} A$ such that $A' \ind_C B'$.
    \item \textit{Symmetry}: For all $A, B, C \subset \M$, $A \ind_C B$ iff $B \ind_C A$.
    \item \textit{Right transitivity}: For all $A \subset \M$ and $D \subseteq C \subseteq B\subset \M$, if $A \ind_C B$ and $A \ind_D C$, then $A \ind_D B$.
    \item \textit{Local character}: For all $A \subset \M$, there is some cardinal $\kappa(A)$ such that, for any $B \subset \M$, there is some $C \subseteq B$ of cardinality $< \kappa(A)$ such that $A \ind_C B$. 
\end{itemize}
We also have left-sided versions for each of the right-sided properties above. We omit the side if both left and right versions of the property hold. We say an independence relation $\ind$ is defined \textbf{over models} if its base can only be a model, i.e., if $A \ind_C B$ is only defined if $C$ is a model of $T$. Given independence relations $\ind$ and $\ind'$, we write $\ind \implies \ind'$ if $A \ind_C B$ implies $A \ind'_C B$. 
\subsection{Rosy and $\NSOP_1$ theories}
In \cite{onshuus2006properties}, Onshuus introduced the notions of \th-forking (read as ``thorn-forking'') and rosy theories, a class that generalises both simple and o-minimal theories. We start by recalling all the relevant definitions from that paper. We follow the conventions in \cite{hoffmann2023thorn}.
\begin{definition}
    Let $\phi(x,y)$ be a formula, $b$ a tuple, and $C$ a set of parameters. 
    \begin{enumerate}[(i)]
        \item We say that $\phi(x,b)$ \textbf{strongly divides over $C$} if the set of images of $\phi(\M,b)$ under elements of $\Aut(\M/C)$ is infinite and $k$-inconsistent for some $k < \omega$.
        \item We say that $\phi(x,b)$ \textbf{\th-divides over $C$} if there is a tuple $c$ such that $\phi(x,b)$ strongly divides over $Cc$. 
        \item We say that $\phi(x,b)$ \textbf{\th-forks over $C$} if there exist formulas $\psi_i(x, c_i)$ for $i < n$ such that $\phi(x,a) \vdash \bigvee_{i < n} \psi_i(x, c_i)$ and each $\psi_i(x,c_i)$ \th-divides over $C$.
        \item We say $A$ is \textbf{\th-independent} from $B$ over $C$ if, for some (equiv., any) enumeration $a$ of $A$, $\tp(a/CB)$ does not contain a formula that \th-forks over $C$. We denote this directly for tuples by $a \thind_C B$.
    \end{enumerate}
\end{definition}
\begin{definition}
    We say a theory $T$ is \textbf{rosy} if $\thind$ in $T^\eq$ satisfies local character. 
\end{definition}
\begin{fact} \label{fact:properties-of-thorn-independence}
    In any theory, $\thind$ satisfies monotonicity, right base monotonicity, left transitivity, and right extension. If $T$ is rosy, then, in $T^\eq$, $\thind$ satisfies in addition symmetry.
\end{fact}
Another notion of independence that has become widespread in the literature appears in the work of Kaplan and Ramsey in \cite{kaplan2020kim}, who introduced the notion of \textit{Kim-independence}.
\begin{definition} \label{def:kind}
    Let $\phi(x,y)$ be a formula, $b$ a tuple, $(b_i)_{i < \omega}$ a sequence, and $M \models T$.
    \begin{enumerate}[(i)]
        \item For a global type $q$, we say $(b_i)_{i < \omega}$ is a \textbf{Morley sequence in $q$ over $M$} if $b_i \models q|_{Mb_{<i}}$ for all $i \in \omega$. If $q$ is $M$-invariant, we call $(b_i)_{i < \omega}$ an \textbf{$M$-invariant Morley sequence}.
        \item We say that $\phi(x,b)$ \textbf{Kim-divides over $M$} if there is some $M$-invariant Morley sequence $(b_i)_{i < \omega}$ with $b_0 = b$ such that $\{\phi(x, b_i) : i < \omega\}$ is inconsistent. 
        \item We say that $\phi(x,b)$ \textbf{Kim-forks over $M$} if there exist formulas $\psi_i(x, c_i)$ for $i < n$ such that $\phi(x,b) \vdash \bigvee_{i < n} \psi_i(x, c_i)$ and each $\psi_i(x,c_i)$ Kim-divides over $M$.
        \item We say $A$ is \textbf{Kim-independent} from $B$ over $M$ if, for some (equiv., any) enumeration $a$ of $A$, $\tp(a/MB)$ does not contain any formula that Kim-forks over $M$. We denote this by $a \kind_M B$.
        \item We say $(b_i)_{i < \omega}$ is \textbf{$\kind$-Morley over $M$} if it is $M$-indiscernible and $b_i \kind_M b_{<i}$ for all $i$.  
    \end{enumerate}
\end{definition}
One of the main results from \cite{kaplan2020kim} characterises the class of $\NSOP_1$ theories, which had been defined combinatorially by D\v{z}amonja and Shelah in \cite{dzamonja2004maximality}, in terms of Kim-independence. We omit the combinatorial definition of $\NSOP_1$, as we will not need it for what follows, but the following result is useful:
\begin{theorem}[\protect{\cite[Theorem 5.16]{kaplan2020kim}}]
    A theory $T$ is $\NSOP_1$ iff Kim-independence is symmetric.
\end{theorem}
\section{The stable Kim-forking conjecture}
In \cite{bossut2024stable}, Bossut introduces the following property of theories:
\begin{definition}
    Let $T$ be a complete $\mathcal{L}$-theory. We say $T$ has the \textbf{stable Kim-forking property over models} if, for any model $M \models T$ and $B \supseteq M$, whenever $p(x) \in S(B)$ Kim-forks over $M$, there is some stable formula $\phi(x,y) \in \mathcal{L}(M)$ and some $b \in B$ such that $\phi(x,b) \in p(x)$ Kim-forks over $M$. 
\end{definition}
\begin{conjecture}[Stable Kim-forking, \protect{cf., \cite[Question 7.3]{baldwin2024simple}}] \label{conj:stable-kim-forking}
    Every $\NSOP_1$ theory has the stable Kim-forking property over models. 
\end{conjecture}
\begin{remark}
    Bossut calls the above property the ``weak stable Kim-forking property,'' in order to distinguish it from a stronger version. However, he shows in \cite[Corollary 3.8]{bossut2024stable} that having the stronger version implies that the theory is simple. Since we focus on the version of the conjecture that applies to (strictly) $\NSOP_1$ theories, we omit this stronger version and follow the terminology in \cite{baldwin2024simple}.
\end{remark}
In what follows, we discuss three of the most studied $\NSOP_1$ examples to show that they all have the stable Kim-forking property, thus providing some positive evidence toward Conjecture \ref{conj:stable-kim-forking}.
\begin{example}
    Let $\mathcal{L}$ be a two-sorted language with sorts $P$ and $O$, and a ternary relation $E \subseteq P \times O^2$. Let $T_{feq}^*$ be the model completion of the theory saying:
    \begin{itemize}
        \item $P$ and $O$ partition the structure.
        \item For all $p \in P$, $E(p, x, y)$ defines an equivalence relation on $O$.
    \end{itemize}
    Chernikov and Ramsey show in \cite[Corollary 6.20]{chernikov2016model} that $T_{feq}^*$ is $\NSOP_1$. Adler proves directly in \cite[Example 2.6]{adler2009geometric} that $T_{feq}^*$ is not rosy. Bossut explains that $T_{feq}^*$ has stable Kim-forking in \cite[discussion after Corollary 4.11]{bossut2024stable}.
\end{example}
\begin{example}
    Let $T_\infty$ be the model companion of the theory of an infinite dimensional vector space over an algebraically closed field with an alternating/symmetric bilinear form in the two-sorted language with a sort for the vector space and another sort for the underlying field.

    Chernikov and Ramsey show that $T_\infty$ is $\NSOP_1$ in \cite[Corollary 6.4]{chernikov2016model}, and Granger shows it is not simple in \cite[Proposition 7.4.1]{granger1999stability}. In \cite{kim2011notions} a proof of non-rosiness is suggested. Finally, by \cite[Proposition 9.37]{kaplan2020kim}, $\kind$ coincides with algebraic independence in $T_\infty$, and so $T_\infty$ has the stable Kim-forking property.
\end{example}
\begin{example}
    We say a field $F$ is \textbf{pseudo-algebraically closed} (or \textbf{PAC}) if every absolutely irreducible variety defined over $F$ has an $F$-rational point.

    Let us denote by $\mathcal{G}(F)$ the absolute Galois group of $F$, viewed as a topological group in the usual fashion. We say $F$ is \textbf{Frobenius} if $\mathcal{G}(F)$ satisfies the \textit{Iwasawa property}, i.e., if $\phi \colon \mathcal{G}(F) \to A$ and $\theta \colon B \to A$ are continuous epimorphisms and $B$ is a finite quotient of $\mathcal{G}(F)$, then there is a continuous epimorphism $\psi \colon \mathcal{G}(F) \to B$ making the following diagram commute:
    \begin{equation*}
        \begin{tikzcd}
            & \mathcal{G}(F) \arrow[d, "\phi"] \arrow[ld, dashed, swap, "\psi"] \\
            B \arrow[r, "\theta"] & A
        \end{tikzcd}
    \end{equation*}
    In particular, whenever $F$ is $e$-free (i.e., $\mathcal{G}(F) \cong \hat{\mathbb{F}}_e$, the free profinite group on $e$ many generators) for some $e \in \omega$, or $\omega$-free (i.e., $\mathcal{G}(F) \cong \hat{\mathbb{F}}_\omega$), it is Frobenius.

    The theory of a Frobenius field $F$ in $\mathcal{L} = \{+, -, \cdot, 0,1\}$ is $\NSOP_1$ (cf., \cite[\S9.3]{kaplan2020kim}). If $F$ is an $e$-free PAC field, it is simple (cf., \cite{hrushovski1991pseudo}), and hence rosy, whereas if it is $\omega$-free it is $\NSOP_1$ (cf., \cite[Corollary 6.2]{chernikov2016model}) and not simple (cf.,  \cite[Theorem 3.9]{chatzidakis1999simplicity}). A proof that the theory of an $\omega$-free PAC field is not rosy is hinted at in \cite[3.5]{chatzidakis2008independence} and written explicitly (and in more generality) in \cite{montenegro2017pseudo}. 
    
    In what follows, we will show that the theory of a Frobenius field has the stable Kim-forking property over models. We use in our proof the notion of an \textit{equation}, first introduced by Srour in \cite{pillay1984closed}. We take as a definition an equivalent characterisation of equations which follows by compactness (cf. \cite{martin2020equational}):
    \begin{definition}
        A formula $\phi(x,y)$ is an \textbf{equation} if, for every indiscernible sequence $(a_i, b_i)_{i < \omega}$, if $\models \phi(a_i, b_j)$ for all $i < j$, then $\models \phi(a_i, b_i)$ for all $i < \omega$.
    \end{definition}
\begin{fact}[\protect{\cite[Proposition 2.6]{pillay1984closed}}] \label{fact:equation-stable}
    If $\phi(x,y)$ is an equation, then it is stable.
\end{fact}
\begin{remark}
    We use the characterisation of Kim-independence for the theory $T$ of a Frobenius field, which comes from \cite[Theorem 9.32]{kaplan2020kim}: for any $\acl$-closed subfields $C \subseteq A,B$ of a sufficiently saturated field $F \models T$, $A \kind_C B$ iff $A \ind_C^{\text{SCF}} B$ and $S\mathcal{G}(A) \find_{S\mathcal{G}(C)} S\mathcal{G}(B)$, where $\find$ is nonforking independence in the $\omega$-sorted theory of the inverse system of $\mathcal{G}(F)$ (for more details, see \cite{chatzidakis1998model}). 

    In \cite[Lemma 6.2]{chernikov2016model}, Chernikov and Ramsey explicitly give the formula witnessing strong finite character for Kim-independence in the theory of an $\omega$-free PAC field, which they attribute to Zoé Chatzidakis. We note here that the same formula proves strong finite character in the theory of any Frobenius field. In what follows, we briefly describe this formula.
    
    Suppose, towards proving stable Kim-forking, that $A \nind_C^{\text{K}} B$ with $A,B,C$ acl-closed and $C \subseteq A,B$ contained in a large saturated field $F \models T$. If $A \nind_C^{\text{SCF}} B$, this is witnessed by an equation and an inequation, which are always stable formulas. So it suffices to consider the case where $A \ind_C^{\text{SCF}} B$ and $S\mathcal{G}(A) \nind_{S\mathcal{G}(C)}^{\text{f}} S\mathcal{G}(B)$. Using finite character and \cite[Proposition 4.1]{chatzidakis1998model}, it follows that we can find open normal subgroups $M_1 \lhd \mathcal{G}(A)$ and $M_2 \lhd \mathcal{G}(B)$ such that, if $N$ is the smallest open normal subgroup of $\mathcal{G}(F)$ containing $M_1$ and $M_2$, then $N \not\subseteq \mathcal{G}(C)$. 
    
    Using some Galois theory and the fact that $F$ is a regular extension of $A$ and $B$, this means that there are some $\alpha \in \langle AC \rangle^{\textnormal{alg}}$ and $\beta \in \langle BC \rangle^\textnormal{alg}$ not in $F$ such that $F(\alpha) = F(\beta)$ is Galois over $F$ and $\beta \notin F\langle C \rangle^\textnormal{alg}$. We may choose some $b' \in \acl(CB)$ such that $\langle CB\beta \rangle \cap F \subseteq \langle CBb' \rangle$ and $\langle CBb' \rangle$ is closed under $\Aut(\acl(CB)/\langle CB \rangle)$. Similarly, we can pick $a' \in \acl(AC)$ such that $\langle CA \alpha \rangle \cap F \subseteq \langle CAa' \rangle$ and $\langle CAa'\rangle$ is closed under $\Aut(\acl(CA)/\langle CA \rangle)$. (Note: $\alpha, \beta, a', b'$ can be taken as singletons by the Primitive Element Theorem.) 

    The idea now is to find a formula that captures the behaviour of $a,b,\alpha, \beta$ as described above. In order to do so, let:
    \begin{itemize}
        \item $P(Y, b, c)$ be the minimal polynomial of $b'$ over $\langle BC \rangle$.
        \item $Q(Z, Y, b, c)$ be such that $Q(Z, b', b, c)$ is the minimal polynomial of $\beta$ over $\langle CB b' \rangle$.
        \item $R(W, a, c)$ be the minimal polynomial of $a'$ over $\langle AC \rangle$.
        \item $S(W, T, a, c)$ be such that $S(W, a', a, c)$ is the minimal polynomial of $\alpha$ over $\langle CAa' \rangle$. 
    \end{itemize}
    We now take $\phi(t,b,c)$ (with variable $t$) to be the formula saying: ``there exists some $y \equiv_{bc} b'$ and a solution $w$ of $R(W, t, c) = 0$ such that (i) $S(X, w, t, c)$ is irreducible over $F$ of degree $[\langle CA\alpha \rangle : \langle CAa' \rangle]$, and (ii) for every solution $z$ of $Q(Z, y,b,c) = 0$, $F(z)$ contains a solution of $S(X, w, t, c) = 0$.'' This can indeed be expressed using a first-order formula (for more details, see Chernikov and Ramsey's paper). In \cite[Lemma 6.2]{chernikov2016model}, it is shown that $\phi(t,b,c)$ Kim-forks over $C$.
\end{remark}
\begin{proposition} \label{prop:omega-free-pac-has-stable-kim-forking}
    The theory of a Frobenius field has the stable Kim-forking property over algebraically closed sets.
\end{proposition}
\begin{proof}
    It is enough to show that the formula $\phi(t, s, c)$ (in variables $t, s$) from the previous remark is stable. We prove something stronger (cf., Fact \ref{fact:equation-stable}): namely, we will show that $\phi(t,s,c)$ is an equation in Srour's sense.

    Suppose $(a_i,b_i)_{i < \omega}$ is a $C$-indiscernible sequence such that $\models \phi(a_i, b_j, c)$ for all $i < j$. Let us choose, for each $i < j$, $d_{ij}$ and $e_{ij}$ to be witnesses to $y$ and $w$ as stipulated by each appropriate instance of $\phi$, and also pick $f_{ij}$ to be a solution to $Q(Z, d_{ij}, b_j, c) = 0$. Since by choice they are all roots of $Q$, it follows that, for all $i < j$, $F(f_{ij})$ is a proper Galois extension of $F$ of degree $[\langle CB\beta \rangle : \langle CBb' \rangle]$ which is not contained in $F\langle C \rangle^\textnormal{alg}$. Note that, by definition of $\phi$, for each $i < j$, $F(f_{ij})$ is precisely the field generated by any of its roots of $S(X, e_{ij}, a_i, c) = 0$ over $F$.
    \begin{nclaim}
        There exist $k < i < j$ such that $e_{ki} = e_{kj}$ and $d_{kj} = d_{ij}$.
    \end{nclaim}
    \begin{proof}[Proof of Claim 1]
        Suppose that $P(Y, b_0, c) = 0$ has $m$ solutions and $R(W, a_0, c) = 0$ has $n$ solutions. By $C$-indiscernibility, we get the same number of solutions for any indices $i,j < \omega$. Thus, our choice of $d_{ij}$ and $e_{ij}$ defines an $mn$-colouring on $\omega^{(2)}$. Hence, by Ramsey's Theorem, there is an infinite subset $X \subseteq \omega$ such that $X^{(2)}$ is monochromatic. Pick $k,i,j \in X$ with $k < i < j$. Then $\{k,i\}, \{k,j\}, \{i,j\} \in X$, so it follows that $e_{ki} = e_{kj}$ and $d_{kj} = d_{ij}$, as required. \hfill \pushQED{$\qed_{\text{Claim 1}}$}
    \end{proof}
    Let $d_j := d_{kj} = d_{ij}$ and $e_k := e_{ki} = e_{kj}$. Then $F(f_{ki})$ and $F(f_{kj})$ both contain solutions to $S(X, e_k, a_k, c) = 0$, and since they are Galois, by our earlier remarks we get $F(f_{ki}) = F(f_{kj})$. Moreover, $F(f_{kj})$ and $F(f_{ij})$ also contain all solutions of $Q(Z, d_j, b_j, c) = 0$, so we get $F(f_{kj}) = F(f_{ij})$. Therefore, $F(f_{ki}) = F(f_{ij})$.
    \begin{nclaim}
        $\models \phi(a_i, b_i, c)$.
    \end{nclaim}
    \begin{proof}
        Assume not, for contradiction. So we can find a solution $f_{ii}$ of $Q(Z, d_{ki}, b_i, c) = 0$ such that $F(f_{ii})$ contains no solutions of $S(X, e_{ij}, a_i, c) = 0$. But note that, by this choice of $f_{ii}$ and the above remarks, it follows that $f_{ii} \in F(f_{ki}) = F(f_{ij})$, and hence, $F(f_{ii})$ contains a solution to $S(X, e_{ij}, a_i, c) = 0$, a contradiction. \pushQED{$\qed_{\text{Claim 2}}$}
    \end{proof}
    It follows from Claim 2 by $C$-indiscernibility of $(a_ib_i)_{i < \omega}$ that this holds for all $i < \omega$, as required.
\end{proof}
\end{example}
\section{On the relation between \texorpdfstring{$\NSOP_1$}{NSOP1} and rosy theories}
A frequent question in neostability theory concerns the relationship between the ``canonical'' independence relations arising from different classes of theories. One example is how $\find$ relates to $\thind$ in simple theories. A partial answer to this question appears in \cite[Corollary 5.1.3]{onshuus2006properties}:
\begin{theorem} \label{thm:simple-thind}
    Suppose that the stable forking conjecture holds. If $T$ is simple, then $\find = \thind$. 
\end{theorem}
More recently, Byungham Kim and Christian d'Elb\'{e}e posed independently (in \cite[Question 6.2]{kim2021weak} and \cite[\S 3.4]{d2021forking}, resp.) the following question:
\begin{question}
    Is there any $\NSOP_1$ rosy theory which is non-simple?
\end{question}
In what follows, we give a partial answer to this question under the assumption of the stable Kim-forking conjecture. This resembles Onshuus' result for simple theories (cf., Theorem \ref{thm:simple-thind}). But first, we need to collect some results from the literature. 

A close inspection of the proof appearing in \cite[Proposition 3.3]{hoffmann2023thorn} shows the following result, first proved by Onshuus in \cite[Theorem 5.1.1]{onshuus2006properties}:
\begin{fact} \label{fact:stable-forking-implies-th-forking}
    Let $M \prec N$ be models of $T = T^\eq$, and $p(x) \in S(N)$. If there is some $\phi(x,n) \in p(x)$ that forks over $M$ such that $\phi(x,y) \in \mathcal{L}(M)$ is stable, then $p(x)$ $\textnormal{\th}$-forks over $M$.
\end{fact}
Recall that $(M_i)_{i < \kappa}$ is a \textbf{continuous chain} of models if $M_i \models T$ for all $i$, $M_i \subseteq M_j$ for all $i < j$, and $M_\alpha = \bigcup_{\beta < \alpha} M_\beta$ for $\alpha$ limit. We assume our continuous chains are always increasing. We will use the following version of local character:
\begin{lemma} \label{lem:rosy-local-character}
    Suppose $T = T^\eq$ is rosy. Then $\thind$ satisfies \textbf{local character over models}, that is: Let $\kappa > \size{T}^+$ be regular, $(M_i)_{i < \kappa}$ be a continuous chain of models of $T$ such that $\size{M_i} < \kappa$ for all $i$, and let $M := \bigcup_{i < \kappa} M_i$ be such that $\size{M} = \kappa$. Let $a$ be a finite tuple. Then there is $i < \kappa$ such that $a \thind_{M_i} M$. 
\end{lemma}
\begin{proof}
    For an abstract proof of this result at the level of independence relations, see, e.g., \cite[Lemma 9.6]{dobrowolski2022kim} (using the properties listed in Fact \ref{fact:properties-of-thorn-independence}).
\end{proof}
We also require some important facts about $\NSOP_1$ theories:
\begin{fact} \label{fact:nsop_1-simple-iff-base-mon}
\begin{enumerate}[(i)]
    \item \cite[Proposition 8.8]{kaplan2020kim} An $\NSOP_1$ theory $T$ is simple iff $\kind$ satisfies base monotonicity over models.
    \item (This is a corollary to \cite[Theorem 6.5]{kaplan2020kim}) Let $T$ be $\NSOP_1$ and $M \models T$. If $a \kind_M b$ and $I = (b_i)_{i < \omega}$ is a $\kind$-Morley sequence over $M$ with $b_0 = b$, then there is $a' \equiv_{Mb} a$ such that $a'b_i \equiv_M ab$ for all $i < \omega$.
\end{enumerate}
\end{fact}
We also adapt the terminology from \cite{dobrowolski2022kim}: given a model $M \models T$ and a sequence $(a_i)_{i < \kappa}$, there exists a continuous chain of models $(M_i)_{i < \kappa}$ such that $M \subseteq M_0$, $a_{<i} \subset M_i$, and $a_i \thind_M M_i$ for all $i < \kappa$. We will call this sequence $(M_i)_{i < \kappa}$ a \textbf{\th-independent chain} for $(a_i)_{i < \kappa}$ over $M$. 
\begin{theorem} \label{thm:stable-kim-forking-nsop1-and-rosy-is-simple}
    If $T$ is $\NSOP_1$ and rosy, and $T^\eq$ has the stable Kim-forking property over models, then $T$ is simple.
\end{theorem}
\begin{proof}
    Let $T$ be $\NSOP_1$ and rosy. Recall (see, e.g., \cite[Remark \& Example 2.3.8]{kim2014simplicity}, which can be also adapted for $\NSOP_1$) that $T$ is $\NSOP_1$ (resp., simple) iff $T^\eq$ is $\NSOP_1$ (resp., simple), so without loss of generality we may assume $T = T^\eq$. In particular, by assumption, $T$ has the stable Kim-forking property over models.
    \begin{claim}
        $\kind = \thind$ over models. 
    \end{claim}
    \begin{proof}
        ($\Leftarrow$) Suppose that $a \ind^{\text{\th}}_M b$ but $a \nind_M^{\text{K}} b$. By the stable Kim-forking property, there is some $\mathcal{L}(M)$-formula $\phi(x,y)$ that is stable such that $\phi(x,b) \in \tp(a/Mb)$ Kim-forks over $M$. By right extension, we can find some $N \supset M \cup \{b\}$ such that $a \ind_M^{\text{\th}} N$ and $N \models T$. Then $\phi(x,b) \in \tp(a/N)$, and in particular $\phi(x,b)$ forks over $M$. Thus, by Fact \ref{fact:stable-forking-implies-th-forking}, $\tp(a/N)$ \th-forks over $M$, a contradiction.

        ($\Rightarrow$) We adapt the proof from \cite[Theorem 9.1]{dobrowolski2022kim} (which is in turn based on the original proof from \cite{chernikov2023transitivity}). So suppose that $a \kind_M b$. By the above remark, we can find a long $M$-indiscernible sequence $(b_i)_{i < \kappa}$ with a \th-independent chain $(M_i)_{i < \kappa}$ of models of $T$. By the previous paragraph, we have $b_i \kind_M M_i$ for all $i < \kappa$. In particular, $(b_i)_{i < \kappa}$ is $\kind$-Morley over $M$, and thus, by Fact \ref{fact:nsop_1-simple-iff-base-mon}(ii), there is $a' \equiv_{Mb} a$ such that $a'b_i \equiv_M ab$ for all $i < \omega$. 

        By monotonicity and downwards Löwenheim-Skolem, we may assume $\kappa = (\size{T} + \size{M})^+$ and $(M_i)_{i \in \kappa}$ is a continuous chain with $\size{M_i} < \kappa$. Let $M_\kappa := \bigcup_{i \in \kappa} M_i$. Since $T$ is rosy, by Lemma \ref{lem:rosy-local-character}, $\thind$ satisfies local character over models. Therefore, there is $i_0 < \kappa$ such that $a' \thind_{M_{i_0}} M_\kappa$, and so by monotonicity, $a' \thind_{M_{i_0}} b_{i_0}$. But also $b_{i_0} \thind_M M_{i_0}$. Hence, by symmetry and transitivity, $a' \thind_M b_{i_0}$, so $a \thind_M b$ by invariance, as required. \hfill \pushQED{$\qed_{\text{Claim}}$}
    \end{proof}
    This equality implies that $\kind$ satisfies base monotonicity over models. Therefore, by Fact \ref{fact:nsop_1-simple-iff-base-mon}(i), $T$ is simple. This concludes the proof.
\end{proof}
\begin{corollary}
    Suppose that the stable Kim-forking conjecture holds. If $T$ is $\NSOP_1$ and rosy, then $T$ is simple.
\end{corollary}
\begin{remark}
    It is an open problem whether, in order to conclude that $T^\eq$ has stable Kim-forking over models, it suffices to prove stable Kim-forking for types over real parameters (see \cite[\S 5]{bossut2024some} for more details).
\end{remark}
Let us note that this pattern fails for higher levels of the $\SOP_n$ hierarchy: for instance, the generic triangle-free graph, which is $\NSOP_4$ but $\SOP_3$, is rosy.

A corollary of this result emphasizes the distinction between simple and $\NSOP_1$ theories. Recall that $T$ is said to be \textbf{pregeometric} if $\acl$ satisfies exchange, i.e., whenever $a \in \acl(Ab) \setminus \acl(A)$, we have $b \in \acl(Aa)$. This large class of theories contains those of strongly minimal and o-minimal theories, plus a plethora of other $\NIP$ examples, such as the theory of $p$-adically closed fields. We 
can also find several examples of simple pregeometric theories, such as simple free amalgamation theories (cf., \cite{conant2017axiomatic}) and pseudofinite fields (cf., \cite{hrushovski1991pseudo}). 

Generally speaking, if $T$ is pregeometric, $T^\eq$ need not be so. However, all that is required to extend the notion of dimension given by algebraic closure to the imaginary case is the weakest form of elimination of imaginaries, namely, \textbf{geometric elimination of imaginaries}, i.e., for any imaginary $e$, there is some real tuple $a$ such that $e \in \acl^\eq(a)$ and $a \in \acl(e)$. This is the main idea underlying the proof of the following result, due to Ealy and Onshuus:
\begin{fact}[\protect{\cite[Theorem 4.12]{ealy2007characterizing}}]
    If $T$ is pregeometric and has geometric elimination of imaginaries, then $T$ is rosy.
\end{fact}
Thus, together with Theorem \ref{thm:stable-kim-forking-nsop1-and-rosy-is-simple}, this implies the following:
\begin{corollary}\label{cor:stable-kf-implies-no-nsop1-pregeometric-theory}
    Suppose the stable Kim-forking conjecture holds. Then there is no strictly $\NSOP_1$ pregeometric theory with geometric elimination of imaginaries.
\end{corollary}

\bibliographystyle{amsalpha}
\bibliography{list}

\providecommand{\bysame}{\leavevmode\hbox to3em{\hrulefill}\thinspace}
\providecommand{\MR}{\relax\ifhmode\unskip\space\fi MR }
\providecommand{\MRhref}[2]{%
  \href{http://www.ams.org/mathscinet-getitem?mr=#1}{#2}
}
\providecommand{\href}[2]{#2}
\begin{thebibliography}{BFM24}

\bibitem[Adl09]{adler2009geometric}
Hans Adler, \emph{A geometric introduction to forking and thorn-forking}, Journal of Mathematical Logic \textbf{9} (2009), no.~01, 1--20.

\bibitem[Bal88]{baldwin1988fundamentals}
John~T. Baldwin, \emph{Fundamentals of stability theory}, Perspectives in Mathematical Logic, Springer-Verlag, Berlin, 1988. \MR{918762}

\bibitem[BFM24]{baldwin2024simple}
John Baldwin, James Freitag, and Scott Mutchnik, \emph{Simple homogeneous structures and indiscernible sequence invariants}, arXiv preprint arXiv:2405.08211 (2024).

\bibitem[Bos24a]{bossut2024stable}
Yvon Bossut, \emph{A note on stable {K}im-forking}, working paper or preprint, November 2024.

\bibitem[Bos24b]{bossut2024some}
\bysame, \emph{On some {F}raisse limits with free amalgamation}, arXiv preprint arXiv:2403.07616 (2024).

\bibitem[Cha98]{chatzidakis1998model}
Zo{\'e} Chatzidakis, \emph{Model theory of profinite groups having the {I}wasawa property}, Illinois Journal of Mathematics \textbf{42} (1998), no.~1, 70--96.

\bibitem[Cha99]{chatzidakis1999simplicity}
\bysame, \emph{Simplicity and independence for pseudo-algebraically closed fields}, London Mathematical Society Lecture Note Series (1999), 41--62.

\bibitem[Cha08]{chatzidakis2008independence}
\bysame, \emph{Independence in (unbounded) {PAC} fields, and imaginaries}, unpublished note (2008).

\bibitem[Che14]{chernikov2014theories}
Artem Chernikov, \emph{Theories without the tree property of the second kind}, Annals of Pure and Applied Logic \textbf{165} (2014), no.~2, 695--723.

\bibitem[CKR23]{chernikov2023transitivity}
Artem Chernikov, Byunghan Kim, and Nicholas Ramsey, \emph{Transitivity, lowness, and ranks in {$\text{NSOP}_1$} theories}, The Journal of Symbolic Logic \textbf{88} (2023), no.~3, 919--946.

\bibitem[Con17]{conant2017axiomatic}
Gabriel Conant, \emph{An axiomatic approach to free amalgamation}, J. Symb. Log. \textbf{82} (2017), no.~2, 648--671. \MR{3663421}

\bibitem[CR16]{chernikov2016model}
Artem Chernikov and Nicholas Ramsey, \emph{On model-theoretic tree properties}, Journal of Mathematical Logic \textbf{16} (2016), no.~02, 1650009.

\bibitem[d'E21]{d2021forking}
Christian d'Elb{\'e}e, \emph{Forking, imaginaries, and other features of $\mathrm{ACFG}$}, The Journal of Symbolic Logic \textbf{86} (2021), no.~2, 669--700.

\bibitem[DK22]{dobrowolski2022kim}
Jan Dobrowolski and Mark Kamsma, \emph{Kim-independence in positive logic}, Model Theory \textbf{1} (2022), no.~1, 55--113.

\bibitem[DS04]{dzamonja2004maximality}
Mirna D{\v{z}}amonja and Saharon Shelah, \emph{On {$\triangleleft^*$}-maximality}, Annals of Pure and Applied Logic \textbf{125} (2004), no.~1-3, 119--158.

\bibitem[EO07]{ealy2007characterizing}
Clifton Ealy and Alf Onshuus, \emph{Characterizing rosy theories}, The Journal of Symbolic Logic \textbf{72} (2007), no.~3, 919--940.

\bibitem[Gra99]{granger1999stability}
Nicolas Granger, \emph{Stability, simplicity and the model theory of bilinear forms}, The University of Manchester (United Kingdom), 1999.

\bibitem[HP23]{hoffmann2023thorn}
Daniel~Max Hoffmann and Anand Pillay, \emph{Thorn forking, weak normality, and theories with selectors}, The Journal of Symbolic Logic \textbf{88} (2023), no.~4, 1354--1366.

\bibitem[Hru91]{hrushovski1991pseudo}
Ehud Hrushovski, \emph{Pseudo-finite fields and related structures}, Model theory and applications \textbf{11} (1991), 151--212.

\bibitem[Kim01]{kim2001simplicity}
Byunghan Kim, \emph{Simplicity, and stability in there}, The Journal of Symbolic Logic \textbf{66} (2001), no.~2, 822--836.

\bibitem[Kim14]{kim2014simplicity}
\bysame, \emph{Simplicity theory}, Oxford Logic Guides, vol.~53, Oxford University Press, Oxford, 2014. \MR{3156332}

\bibitem[Kim21]{kim2021weak}
\bysame, \emph{Weak canonical bases in {$\text{NSOP}_1$} theories}, The Journal of Symbolic Logic \textbf{86} (2021), no.~3, 1259--1281.

\bibitem[KK11]{kim2011notions}
Byunghan Kim and Hyeung-Joon Kim, \emph{Notions around tree property 1}, Annals of Pure and Applied Logic \textbf{162} (2011), no.~9, 698--709.

\bibitem[KR20]{kaplan2020kim}
Itay Kaplan and Nicholas Ramsey, \emph{On {K}im-independence}, Journal of the European Mathematical Society \textbf{22} (2020), no.~5, 1423--1474.

\bibitem[KR21]{kaplan2021transitivity}
\bysame, \emph{Transitivity of {K}im-independence}, Advances in Mathematics \textbf{379} (2021), 107573.

\bibitem[KRS19]{kaplan2019local}
Itay Kaplan, Nicholas Ramsey, and Saharon Shelah, \emph{Local character of {K}im-independence}, Proceedings of the American Mathematical Society \textbf{147} (2019), no.~4, 1719--1732.

\bibitem[Mon17]{montenegro2017pseudo}
Samaria Montenegro, \emph{Pseudo real closed fields, pseudo p-adically closed fields and {$\text{NTP}_2$}}, Annals of Pure and Applied Logic \textbf{168} (2017), no.~1, 191--232.

\bibitem[MPZ20]{martin2020equational}
Amador Martin-Pizarro and Martin Ziegler, \emph{Equational theories of fields}, The Journal of Symbolic Logic \textbf{85} (2020), no.~2, 828--851.

\bibitem[Ons03]{onshuus2003th}
Alf Onshuus, \emph{Th-forking, algebraic independence and examples of rosy theories}, arXiv preprint math/0306003 (2003).

\bibitem[Ons06]{onshuus2006properties}
\bysame, \emph{Properties and consequences of thorn-independence}, The Journal of Symbolic Logic \textbf{71} (2006), no.~1, 1--21.

\bibitem[PS84]{pillay1984closed}
Anand Pillay and Gabriel Srour, \emph{Closed sets and chain conditions in stable theories}, The Journal of symbolic logic \textbf{49} (1984), no.~4, 1350--1362.

\bibitem[She78]{shelah1978classification}
Saharon Shelah, \emph{Classification theory and the number of nonisomorphic models}, Studies in Logic and the Foundations of Mathematics, vol.~92, North-Holland Publishing Co., Amsterdam-New York, 1978. \MR{513226}

\end{thebibliography}
\end{document}